\documentclass[11pt]{amsart}
\usepackage{amsmath,amssymb,amsthm}
\usepackage{color,graphics,srcltx}
\usepackage{bm}

\makeatletter
\newtheorem{theorem}{Theorem}[section]
\newtheorem{lemma}[theorem]{Lemma}
\newtheorem{corollary}[theorem]{Corollary}

\newtheorem{remark}[theorem]{Remark}

\theoremstyle{definition}

\newtheorem{example}[theorem]{Example}

\makeatother
\def\Id{I}

\def\Im{\mathop{\mathrm{Im}}\nolimits}

\def\diag{\mathop{\mathrm{diag}}\nolimits}
\def\rk{\mathop{\mathrm{rk}}\nolimits}
\def\tr{{\rm tr}\,}
\def\FF{{\mathbb F}}

\def\ZZ{{\mathbb Z}}

\def\Lin{\mathop{\rm Lin}}

\def\C{\mathop{\mathcal{C}}}
\def\path{\mathop{\raise.7mm\hbox{\vrule width 3mm height .2mm}}}

\begin{document}

\title{On maximal distances in a commuting graph}

\author[Dolinar]{Gregor Dolinar}
\address[Gregor Dolinar]{Faculty of Electrical Engineering, University of Ljubljana,
Tr\v{z}a\v{s}ka cesta 25, SI-1000 Ljubljana, Slovenia.}
 \email[Gregor Dolinar]{gregor.dolinar@fe.uni-lj.si}
  \author[Kuzma]{Bojan Kuzma}
\address[Bojan Kuzma]{${}^1$University of Primorska,
Glagolja\v{s}ka 8, SI-6000 Koper, Slovenia, \and ${}^2$IMFM, Jadranska 19, SI-1000
Ljubljana, Slovenia.}
 \email[Bojan Kuzma]{bojan.kuzma@famnit.upr.si}
 \author[Oblak]{Polona Oblak}
\address[Polona Oblak]{Faculty of Computer and Information Science, University of Ljubljana,
Tr\v{z}a\v{s}ka cesta 25, SI-1000 Ljubljana, Slovenia.}
 \email[Polona Oblak]{polona.oblak@fri.uni-lj.si}

\thanks{The work is partially supported by  a grant from the
Ministry of Higher Education, Science and Technology, Slovenia.}

\begin{abstract}
We study maximal distances in the commuting graphs of  matrix algebras defined over
algebraically closed fields. In particular, we show that the maximal distance can be
attained only between two  nonderogatory matrices. We also describe rank-one and semisimple matrices using the
distances in the commuting graph.
\end{abstract}
\keywords{Algebraically closed field, Matrix algebra, Centralizer, Commuting graph,
Distance, Path, Minimal matrix, Maximal Matrix.}

 \subjclass[2000]{15A27, 05C50, 05C12}

\maketitle

\section{Introduction and preliminaries}

One of the options how to study properties in certain non-commutative algebraic domains
is a commutator. For example, in  algebras the additive commutator, i.e., Lie product
$[A,B]_a=AB-BA$ is usually used and with its help some beautiful results were obtained.
Let us only mention the famous Kleinecke-Shirokov Theorem~\cite{Kleinecke,Shirokov}. In
groups   the multiplicative commutator $[A,B]_m=A^{-1}B^{-1}AB$ is used, and it is a
central tool in studying solvability of groups and hence in Galois theory of solvability
of equations by radicals~\cite{Grove}.

Additional information about non-commuting elements is obtained  by studying the
properties of a commuting graph. For example, if the commuting graphs over two finite
semisimple rings are isomorphic, then their noncommutative parts are also
isomorphic~\cite{AGHM04}. Let us remark that commuting graph can also be used  in
algebraic domains where commutator is not available, e.g., in semigroups or in semirings.

Up until now, one of the prime concerns when studying commuting graphs was calculating
its diameter~\cite{Abdollahi,AMRR06,AraKinKon10,DolObl10,GiuPope10,Mohammadian10,
Segev01}. It turned out that we obtain essentially different results if the
matrix algebra $M_n(\FF)$ is defined over an algebraically closed field $\FF$ than if it is defined over
non-closed one. While in the former case the diameter is always equal to four, provided $n\ge 3$,
in the later case the graph may be disconnected, and if it is connected
the diameter is known to be at most six. The hypothesis is that if the commuting graph is connected, its diameter
is at most five~\cite[Conjecture~18]{AMRR06}. Note that for $n=2 $ the commuting graph over any field is disconnected \cite[Remark 8]{AR06}.

In the present paper we are interested in commuting graphs of matrix algebras $M_n(\FF)$ over
algebraically closed fields $\FF$ with $n \ge 3$.
In particular we  study the maximal distances  between its
vertices. It was already proved in~\cite[Proof of Theorem~3]{AMRR06} that an
elementary Jordan matrix is  always at the maximal distance (i.e., four) from its
transpose. In the present paper we  show that the maximal distance cannot be achieved when
one of the matrices is derogatory. However, if both $A$ and $B$ are non-derogatory we
construct an invertible matrix $S$ so that $A$ and $S^{-1}BS$ are at the distance four.
We also show that there exist an infinite collection
of matrices, pairwise at the maximal distance. Next, we describe rank-one
matrices as the ones which are not at the maximal distance from any derogatory matrix. A
similar result classifies semisimple (i.e., diagonalizable) matrices. Our paper concludes
with a specific example of matrix algebra over algebraically non-closed field, such that the diameter of its commuting
graph is greater than four.

Let us briefly recall some standard definitions and notations. Unless explicitly stated
otherwise, $\FF$ is an algebraically closed field of an arbitrary characteristics.
Further, $M_{m, n}(\FF)$ is the space of $m\times n$ matrices over $\FF$ with a standard
basis $E_{ij}$, and $M_n(\FF)=M_{n, n}(\FF)$ is the matrix algebra with identity~$\Id$.
Let $e_1,\dots,e_n$ be the standard basis of column vectors in $\FF^n$ (i.e., of $n\times
1$ matrices). Given an integer $k\ge 2$  denote by $J_k(\mu)=\mu\Id_k+\sum_{i=1}^{k-1}
E_{i(i+1)}\in M_k(\FF)$ the upper-triangular elementary Jordan cell with $\mu$ on its
main diagonal, and  let $J_1(\mu)= \mu \in \FF$. We  write shortly $J_k = J_k(0)$. Matrix
$B$ is a conjugated matrix of $A$ if $B=S^{-1}AS$ for some invertible matrix $S$. As
usual, $A^{\tr}$ is a transpose of $A\in M_n(\FF)$ and $\rk A$ its rank.

For a matrix algebra $M_n(\FF)$ over a field $\FF$ its commuting graph $\Gamma(M_n(\FF))$
is a simple graph (i.e., undirected and loopless), with the vertex set consisting of all
non-scalar matrices. Two vertices $X, Y$ form an edge  $X \path  Y$ if the
corresponding matrices are different and commute, i.e., if $X \ne Y$ and  $XY=YX$. The
sequence of successive connected vertices $X_0\path  X_1$, $ X_1 \path  X_2$, ...,
$X_{k-1} \path X_{k}$ is  a path of length $k$ and is denoted by $X_0 \path  X_1 \path
\ldots \path X_k$. The distance $d(A,B)$ between  vertices $A$ and $B$ is the length of
the shortest path between them. The diameter of the graph is the maximal distance between
any two vertices of the graph.

Given a subset $\Omega \subseteq M_n(\FF)$, let
$$\C(\Omega)=\{X\in M_n(\FF);\;\;AX=XA \mbox{ for every } A \in \Omega \}$$
be its centralizer. If $\Omega = \{ A \}$ then we write shortly $\C(A) = \C(\{ A \})$. In
graph terminology, the set of all non-scalar matrices from the centralizer of $A$ is equal to the set of all vertices $X$ such that
$d(A,X) \le 1$.  Note that $\FF\Id\in \C(A)$ for any matrix~$A$ and that, by  a double centralizer theorem,
$\C(\C(A))=\FF[A]$ (see~\cite[Theorem~2, pp.~106]{Wed} or~\cite{lagerstrom}).
We remark that in
different articles  a centralizer is also called  a commutant and  is denoted by
$A'=\C(A)$.

A centralizer induces two natural relations on $M_n(\FF)$. One is
the equivalence relation, defined by $A\sim B$ if $\C(A)=\C(B)$.
We  call any such two matrices equivalent. The other relation is a
preorder given by $A\prec B$ if $\C(A)\subseteq \C(B)$. It was
already observed that  minimal  and  maximal matrices in this
poset are of special importance, see for
example~\cite{Dolinar_Semrl,Semrl,Dol-Kuz}. Recall that a matrix
$A$ is minimal if $\C(X) \subseteq \C(A)$ implies $\C(X)=\C(A)$.
It was shown in~\cite[Lemma 3.2]{Semrl} that the matrix $A$ is
minimal if and only if it is nonderogatory, which means that each
of its eigenvalue has geometric multiplicity one, which is further
equivalent to the fact that its Jordan canonical form is  equal to
$J=J_{n_1}(\lambda_1)\oplus\dots\oplus J_{n_k}(\lambda_k)$, with
$\lambda_i\ne\lambda_j$ for $i\ne j$. In this case,
$$\C(J)=\FF[J_{n_1}(\lambda_1)]\oplus\dots\oplus\FF[J_{n_k}(\lambda_k)]=\FF[J],$$
where $\FF[X]$ is an $\FF$-algebra generated by $X$, see~\cite[Theorem~1, pp. 105]{Wed}
or~\cite[Theorem~3.2.4.2]{Horn-Johnson}.

Recall also that a non-scalar matrix $A$ is maximal if $\C(A) \subseteq
\C(X)$ implies $\C(A)=\C(X)$ or $X$ is a scalar matrix. It is known (see~\cite[Lemma~4]{Dolinar_Semrl} and also~\cite[Lemma~3.1]{Semrl}) that
a matrix is maximal if and only if it is equal to $\alpha\Id+\beta P$ or $\alpha\Id+\beta N$,
where $P^2=P$ is a non-scalar idempotent, $N\ne0$ is square-zero (i.e., $N^2=0$), and a scalar $\beta$ is
nonzero. It should be noted that the proof of this fact was done only for the field of
complex numbers, but can be repeated almost unchanged in an arbitrary algebraically closed field.

\section{Results}

Throughout this section, with an exception of the last example, $\FF$ is an algebraically closed field and $n \ge 3$.
We start with three technical lemmas which will be needed in the sequel. First we observe
that every matrix commutes with a rank-one matrix.

\begin{lemma}\label{lem:rk-1'=all}%
For every matrix $A \in M_n(\FF)$ there exists a rank-one matrix $R\in M_n(\FF)$ with $d(A,R) \le 1$.
\end{lemma}
\begin{proof}
Given any $A \in M_n(\FF)$, it suffices to show that~$A$ commutes with
at least one matrix of rank one. Since $\FF=\overline{\FF}$, the
matrix~$A$ has at least one eigenvalue $\lambda$.
So, we may assume without
loss of generality that $A$ is singular, otherwise we would
consider $A - \lambda \Id$. Now, let $ x $ and $ y
$ be nonzero vectors in the kernels of $A$ and $A^{\tr}$,
respectively. Then, $R= x  y ^{\tr}$ is a rank-one matrix with
$AR=(A x ) y ^{\tr}=0= x (A^{\tr} y )^{\tr}=RA$.
\end{proof}

Using Lemma~\ref{lem:rk-1'=all} we can give an alternative proof of the already known
fact about the diameter of a commuting graph~\cite[Corollary~7]{AMRR06}.
\begin{corollary}\label{cor:maximal-distance-in-graph=4}
The distance between any two matrices in the commuting graph is at most four.
\end{corollary}
\begin{proof}
Let $A$ and $B$ be arbitrary matrices. By Lemma~\ref{lem:rk-1'=all} there exist rank-one
matrices $R_1= xf^{\tr}\in \C(A)$ and $R_3= y  g^{\tr}\in \C(B)$.
Since $n\ge 3$ we can find a nonzero $z\in\FF^n$ with $f^{\tr}z=0=g^{\tr}z$ and  a nonzero  $h\in\FF^n$
with $ h^{\tr}x =0=h^{\tr}y $. Then for a rank-one matrix $R_2= x h^{\tr}$ we obtain
 $A=R_0\path  R_1\path  R_2\path R_3\path R_4=B$.
\end{proof}

\begin{lemma}\label{lem:splosna-o-preseku-komuatntov}%
Let~$A=J_{k_1}\oplus J_{k_2}\in M_{k_1+k_2}(\FF)$ be a nilpotent matrix with two Jordan cells
of sizes~$k_1,k_2\ge 1$. Then $d(A,R) \le 2$ for an arbitrary rank-one $R\in M_{k_1+k_2}(\FF)$.
\end{lemma}
\begin{proof}
If $k_1=k_2=1$ then $A$ is a zero matrix and the conclusion is then imminent. Otherwise,
$k_1\ge2$ or $k_2\ge 2$. Let $k=k_1 + k_2$. It is elementary that the matrix $Z=x_1
E_{1k_1} +x_2 E_{1k}+x_3E_{(k_1+1)k_1}+x_4E_{(k_1+1)k}$ commutes with $A$ for any choice
of $x_1,x_2,x_3,x_4\in\FF$. Actually, $ZA = AZ = 0$. Moreover, the matrix $Z$ is non-scalar,
except when~$x_1=x_2=x_3=x_4=0$. Therefore, it suffices to show that for an arbitrary
rank-one matrix $R$ there exist $x_1,\dots,x_4 \in \FF$, such that at least one of them
is nonzero and $ZR=RZ = 0$. To this end, write $R= a  b ^{\tr}$ for some column vectors
$ a =(a_1,\dots,a_k)^{\tr}$ and $ b =(b_1,\dots,b_k)^{\tr}$. Then $ZR=RZ=0$ is
equivalent to $Z a =Z^{\tr} b =0$, hence we must solve a homogeneous system of four linear
equations
\begin{equation}\label{eq:2jordan-blocks-commute-wth-rk1}
\begin{aligned}
x_1 a_{k_1}+ x_2 a_k &=0,\\
x_3 a_{k_1} +x_4 a_k &=0,\\
x_1 b_1+ x_3 b_{k_1+1}&=0,\\
x_2 b_1 +x_4 b_{k_1+1}&=0,
\end{aligned}
\end{equation}
$x_1,x_2,x_3,x_4$ unknown. The corresponding matrix of coefficients is equal to
$$\begin{bmatrix}
 a_{k_1} & a_k & 0 & 0 \\
 0 & 0 & a_{k_1} & a_k \\
 b_1 & 0 & b_{k_1+1} & 0 \\
 0 & b_1 & 0 & b_{k_1+1}
\end{bmatrix}$$
and it is easy to check that it is always singular. Therefore the system~(\ref{eq:2jordan-blocks-commute-wth-rk1})
has a nontrivial solution. This solution defines a non-scalar matrix $Z$, which commutes
with $A$ and $R$, so $d(A,R) \le 2$ in $\Gamma(M_{k_1+k_2}(\FF))$.
\end{proof}
\begin{lemma}\label{lem:A'capR'}%
Suppose $A$ is not minimal. Then $d(A,R) \le 2$ for an arbitrary rank-one matrix $R\in M_n(\FF)$.
\end{lemma}
\begin{proof}
Using conjugation we might assume~$A$ is already in its Jordan form. Since it is not
minimal, hence it is derogatory, at least two Jordan cells contain the same eigenvalue.
Let $k_1, k_2\ge 1$ be their sizes. Define also $k=k_1+k_2$. Moreover,
$\C(A)=\C(A-\lambda\Id)$ so we may also assume that these two Jordan cells are nilpotent
and that $A=J_{k_1}\oplus J_{k_2}\oplus \tilde{A}$.  It is elementary that
$$\C\bigl(J_{k_1}\oplus J_{k_2}  \bigr)\oplus (\FF\Id_{n-k})\subseteq \C(A).$$

Now, let $R= x  y ^{\tr}$ be an arbitrary rank-one matrix.
Decompose $ x = x _1\oplus  x _2\in\FF^{k}\oplus
\FF^{n-k}$ and $ y = y _1\oplus  y _2\in\FF^{k}\oplus
\FF^{n-k}$. We claim that there exists a non-scalar matrix
$\widehat{Z} \in \C(J_{k_1}\oplus J_{k_2} )$ satisfying simultaneously
$\widehat{Z} x _1=\lambda  x _1$ as well as $\widehat{Z}^{\tr} y
_1=\lambda y _1$ for some $\lambda\in\FF$. In fact, this is
trivial when $ x_1= y_1=0$. Otherwise we let
$$\widehat{R}=\begin{cases}
 x _1 y _1^{\tr};&  x _1, y _1\ne0\\
 e _1 y _1^{\tr}; &  x _1=0\\
 x _1 e _1^{\tr}; &  y _1=0
\end{cases}\in M_{k}(\FF),$$
where $ e _1\in \FF^{k}$  is the first vector of the standard basis. By
Lemma~\ref{lem:splosna-o-preseku-komuatntov} there exists at least one non-scalar matrix
$\widehat{Z}\in M_{k}(\FF)$ which commutes  with $\widehat{R}$ as well as with
$J_{k_1}\oplus J_{k_2}\in M_{k}(\FF)$. Therefore, if $ x _1, y _1\ne0$, then $\widehat{Z} x
_1 y _1^{\tr}= x _1(\widehat{Z}^{\tr} y _1)^{\tr}$ and we obtain $\widehat{Z} x _1=\lambda x _1$,
and $\widehat{Z}^{\tr} y _1=\lambda  y _1$ for some $\lambda \in \FF$. If $x_1 = 0$, then
similarly as above $\widehat{Z} e _1=\lambda e _1$, and $\widehat{Z}^{\tr} y _1=\lambda  y _1$.
Obviously $\widehat{Z} x_1=\lambda x_1$. Likewise we argue  if $ y _1=0$.

With the help of~$\widehat{Z}$ we define $Z=\widehat{Z}\oplus
\lambda\Id_{k}\in\C(J_{k_1}\oplus J_{k_2}  )\oplus (\FF\Id_{n-k})\subseteq
\C(A)$. Clearly, $Z x =\widehat{Z} x _1\oplus \lambda x _2=\lambda x $, and similarly,
$Z^{\tr} y =\lambda y $, so~$Z$ commutes with $R= x  y ^{\tr}$ and with $A$.
\end{proof}

 Akbari, Mohammadian, Radjavi, and Raja
proved in~\cite[Lemma 2]{AMRR06} that, for matrices of size $n\ge 3$, the diameter of the
commuting graph is at most four (see also Corollary~\ref{cor:maximal-distance-in-graph=4}
above) and that $d(J,J^{\tr})=4$, thus showing that the diameter of the commuting graph
of matrix algebra over algebraically closed fields is equal to four. It is
well-known~\cite[p.~134]{Horn-Johnson} that the transpose of a matrix is conjugate to the
original, so \cite[Lemma 2]{AMRR06} implies that the maximal distance from $J$ to some of
its conjugates is equal to four. Our next lemma will strengthen their result by
considering maximal distances between an arbitrary minimal matrix $A\in M_n(\FF)$ and
matrices from conjugation orbit $\{S^{-1}BS;\;\;S\hbox{ invertible}\}$ of another minimal
matrix $B\in M_n(\FF)$. Recall that a minimal matrix is conjugate to
$\bigoplus\limits_{i=1}^{k} J_{n_i}(\lambda_i), $ where $\lambda_i \ne \lambda_j$ for $i
\ne j$, and where $(n_1,n_2,\ldots,n_k)$ is a partition of $n$. We will  show below that
for any two given partitions of $n$, we can find two minimal matrices with their Jordan
forms corresponding to these two partitions, at  distance four. One of the matrices is
already in its Jordan canonical form, while the other  is a matrix, conjugated to its
Jordan canonical form by an invertible matrix with all of its minors  nonzero. Such
invertible matrix is for example a Cauchy matrix $\big[\frac{1}{x_i-y_j}\big]_{ij}$ (see
\cite{Schechter}).

\begin{theorem}\label{thm:d=4}%
 Let $S$ be any matrix with all of its minors  nonzero. For any two minimal matrices
 $A=\bigoplus\limits_{i=1}^{k} J_{n_i}(\lambda_i)\in M_n(\FF)$
 and $B=\bigoplus\limits_{i=1}^{l} J_{m_i}(\mu_i)\in M_n(\FF)$,
 we have  $d(A ,S^{-1} B S)=4$.
\end{theorem}

\begin{proof}
  Assume erroneously that  $A$ and $B$, as defined in Lemma, are not at the maximal distance, i.e.,
   $d(A,S^{-1} B S) \leq 3$.
  Since $\C(A)=\C(\alpha A)$ for all nonzero $\alpha \in \FF$, we can lengthen every path by adding vertices
 which correspond to scalar multiples of matrices. So, there  exists a path $A\path X\path Y\path S^{-1} B S$ of length $3$
  in
  $\Gamma(M_n(\FF))$. We can assume
  without loss of generality that $X$ and $Y$ are maximal matrices. Namely, if  $X$ is not
  maximal, then
  there exists a maximal $X' \succ X$, and since $A,Y \in \C(X) \subseteq \C(X')$, we could consider a path
  $A\path X'\path Y\path S^{-1} B S$ of length 3. Likewise for $Y$.

  We will show that no two maximal matrices $X\in \C(A)$ and $Y\in \C(S^{-1} B S)$  commute
  and thus obtain  a contradiction to the assumption    $d(A,S^{-1} B S) \leq 3$.
  Since all maximal matrices are equivalent either to a square-zero matrix or to an idempotent matrix we will
  consider three  cases.

  First, let us assume that both $X$ and $Y$ are square-zero but nonzero.
  Since $A=\bigoplus_{i=1}^{k} J_{n_i}(\lambda_i)$
 and $B=\bigoplus_{i=1}^{l} J_{m_i}(\mu_i)$ are minimal,
  so $\lambda_i\ne \lambda_j$ and $\mu_i\ne\mu_j$ for $i\ne j$, we have
  that $$X=T_1\oplus T_2 \oplus\ldots\oplus T_k\quad\hbox{ and }\quad Y=S^{-1}(T'_1\oplus T'_2 \oplus\ldots\oplus T'_l)S,$$
  where all $T_i$ and $T'_j$ are upper triangular Toeplitz matrices.
  Clearly then $\Im X=\Lin \{e_{\sigma(1)},e_{\sigma(2)},\ldots,e_{\sigma(r)}\}$ for
  some permutation  $\sigma$ of length $n$ and integer $r$,
 $1\leq r \leq \frac{n}{2}$.
  Moreover,  by the block-Toeplitz  structure of $X\ne0$ there exist indices $t$ and $s$ such that  $Xe_t=\alpha e_s\ne0$.
  For the sake of simplicity let us denote $T=T'_1\oplus T'_2 \oplus\ldots\oplus T'_l$. Now, if $YX=XY$, we would have that
  $S^{-1}TSX e_t \in \Im X$. This would imply,
  $$\alpha T S e_s  \in
    S(\Im X)=\Lin\{Se_{\sigma(1)},Se_{\sigma(2)},\ldots,Se_{\sigma(r)}\}$$
  which  is clearly possible
   if and only if the rank of the $n\times r$ matrix
  $M=\left[ \frac{1}{\alpha}Se_{\sigma(1)},\frac{1}{\alpha}Se_{\sigma(2)},\ldots,\frac{1}{\alpha}Se_{\sigma(r)}
    \right]$
  is the same as the rank of the augmented matrix $\left[ M \, \vert \, TSe_s\right]$. However, we will show that
  this is not the case. Since all  minors of $S$ are nonzero, its $s$-es column $Se_s$ has no zero entries and as such cannot
  be annihilated by a nonzero block-Toeplitz matrix $T$.
  Note that $T$ is  also square-zero and so it has at least $ \frac{n}{2}$ zero rows.
  Recall that $r\le \frac{n}{2}$, consequently there exists an $(r+1)\times (r+1)$ submatrix of the augumented
  matrix, having in the
  last column exactly $r$ zeros and one nonzero element. By expanding this $(r+1)\times (r+1)$  minor
  by the last column, we observe that it is equal to a multiple of an $r \times r$ minor of matrix $M$
  which is
  equal to $(\frac{1}{\alpha})^r$ times an $r \times r$ minor of $S$. By the assumption, every minor of $S$ is nonzero and so
    $r+1=\rk \left[ M \, \vert \, TSe_s\right] > \rk M=r$. This  implies
   $TS e_s \notin S(\Im X)$, a contradiction.

 Second, suppose   a non-scalar idempotent $X\in\C(A)$ commutes with a non-scalar square-zero $Y\in\C(S^{-1}BS)$. Without loss of generality,
 $r=\rk X \leq \frac{n}{2}$, otherwise take $\Id-X$ instead of $X$. So,
 $X=\sum_{i=1}^r E_{\sigma(i)\sigma(i)} $ and
  $Y=S^{-1}TS$,   where $\sigma$ and $T=T'_1\oplus T'_2 \oplus\ldots\oplus T'_l$ are as above.
  Define  $t=\sigma(1)$.  Similarly as before, if $YX=XY$ we would have that
  $S^{-1}TSXe_t\in \Im X=\Lin \{e_{\sigma(1)},e_{\sigma(2)},\ldots,e_{\sigma(r)}\}$,
  or, equivalently, $TSe_t\in \Lin \{Se_{\sigma(1)},Se_{\sigma(2)},\ldots,Se_{\sigma(r)}\}$. We proceed
  as in the first case to obtain a contradiction.

  By the symmetry the only case remaining is the case when
  $X$ and $Y$ are both non-scalar idempotents. Write
  $X=\sum_{i=1}^r E_{\sigma(i)\sigma(i)} $ and
  $Y=S^{-1} P S$ for $P=\sum_{i=1}^s E_{\tau(i)\tau(i)}$.
   Without loss of generality, $r,s \leq \frac{n}{2}$, since otherwise we would substitute
   $X$  by $\Id-X$ or  $Y$ by $\Id-Y$.
  Again,   take  $t=\sigma(1)$. If $YX=XY$ then
  $S^{-1}PSXe_t\in \Im X=\Lin \{e_{\sigma(1)},e_{\sigma(2)},\ldots,e_{\sigma(r)}\}$,
  or, equivalently, $PSe_t\in \Lin \{Se_{\sigma(1)},Se_{\sigma(2)},\ldots,Se_{\sigma(r)}\}$.
  Since $\rk P \leq \frac{n}{2}$, it follows that the vector $PSe_t$ has at least $\frac{n}{2}$ zero
  entries. Note that $\frac{n}{2} \geq r$ and $Se_t$ is the $t$-th column of $S$, so it has no zero entries. This
   gives
  $PSe_t\ne0$, a contradiction as in the first case.

This shows $d(A,S^{-1} B S) \geq 4$. But the  diameter of commuting graph is equal to four (see~\cite[Lemma~2]{AMRR06}),
hence $d(A,S^{-1} B S) =4$.
\end{proof}

\begin{remark}The matrix $S^{-1}BS$ from Theorem~\ref{thm:d=4} can be rather complicated. In a special case, when
$A$ is nilpotent we can take $A=J_n$ to achieve that $d(J_n,B)=4$ for any  companion
matrix $B$ in the lower-triangular form. This can be seen by
 a slight adaptation of the proof of~\cite[Lemma 2]{AMRR06}. For convenience we sketch the main points of the
proof. First, it suffices to prove that each maximal $D\in\C(J_n)=\FF[J_n]$ satisfies
$\C(B)\cap\C(D)=\FF\Id$. We may further assume $D=\sum_{i=r}^{n-1}d_i J_n^i$,
$\frac{n}{2}\le r\le n-1$, is square-zero and hence write it as a $3\times 3$ block
matrix with  block at position $(1,3)$ being invertible upper-triangular Toeplitz of size
$(n-r)\times (n-r)$, while all the rest blocks are zero. If $Z\in\C(B)\cap\C(D)$ then in
particular it commutes with $D$. By direct computation using block-matrix structure we
see that the $(n,1)$ entry of $Z$ is zero. However, $Z\in\C(B)$ and since companion
matrices are nonderogatory we have $Z=\sum_{i=0}^{n-1}\lambda_i B^i$. By considering the
images of basis vectors we see that $B^i=\left[\begin{matrix}
 0_{i, (n-i)} & \bigstar_{i, i}\\
\Id_{n-i}& \bigstar_{(n-i), i}
\end{matrix}\right]$. Since $(n,1)$ entry of $Z$ is $0$ we see that
$\lambda_{n-1}=0$. Proceeding inductively we see that $\lambda_i=0$ for every
$i=(n-1),\dots,1$, whence $Z$ is scalar.
\end{remark}

By Theorem~\ref{thm:d=4} there exist different types of matrices which are at the maximal
distance. Next we show that we can find infinitely many  matrices which are in  the
commuting graph pairwise at the maximal distance. Actually, we find an induced graph
which is a tree with an internal vertex and all of its leaves at distance two from the
internal vertex.

\begin{theorem}
There exist an infinite
family of matrices
$(X_\alpha)_\alpha\in M_n(\FF)$ and a rank-one  matrix $Z$ such that
$d(X_\alpha,X_\beta)=4$ for $\alpha\ne \beta$ and $d(X_\alpha,Z)=2$
for all $\alpha$.
\end{theorem}
\begin{proof}
We consider three cases separately.

{\bf Case  \bm{$n=3$}}. Choose $Z=E_{11}$ and let the infinite family consist
of rank one
nilpotent matrices
$$R_\alpha=(0,1,\alpha)^{\tr}(0,\alpha,-1), \quad \alpha \in \FF.$$
It is easy to see that each member commutes with $E_{11}$ and that the elements of the family are pairwise at distance two.
For each index
$\alpha \in \FF$ choose a nilpotent $X_\alpha$ such that $X_\alpha ^2=R_\alpha
$. Since $n=3$
all non-scalar matrices, which commute  with  $X_\alpha$ are equivalent to $X_\alpha$ or to   $X_
\alpha^2=R_\alpha$. Therefore, as
$d(X_\alpha^2,X_\beta^2)=2$ for $\alpha\ne\beta$, we see that $d(X_
\alpha,X_\beta)=4$ for
$\alpha\ne\beta$.

  {\bf Case \bm{$n=4$}}.  Choose $\lambda\in\FF\setminus\{0,1\}$. For nonzero
$\alpha\in\FF$ consider rank-one nilpotent matrix
$N_\alpha=(0,\lambda,\lambda\alpha,\lambda)^{\tr}(0,-\alpha,1,0)$ and rank-one idempotent
$P_\alpha=(0,1,\alpha,0)^{\tr}(0,1,0,-1).$
It is a straightforward calculation that all these matrices are pairwise non-commutative but they all commute with $E_{11}$, hence
\begin{equation}\label{eq:n=4}
d(N_\alpha,N_\beta)=d(P_\alpha,P_\beta)=d(N_\alpha,P_\beta)=2
\end{equation}
for every $\alpha\ne \beta\in\FF\setminus\{0\}$. Moreover, there exists a conjugation
such that $S_\alpha^{-1} N_\alpha S_\alpha=E_{13}$ and $S_\alpha^{-1} P_\alpha
S_\alpha=E_{44}$, for example, take
$$S_\alpha=\begin{bmatrix}
 0 & 1 & 0 & 0 \\
 \lambda  & 0 & 0 & -1 \\
 \alpha  \lambda  & 0 & 1 &
   -\alpha  \\
 \lambda  & 0 & 0 & 0
\end{bmatrix}.$$
Then, for each
$\alpha$ we can find a minimal matrix $X_\alpha=S_\alpha (J_3\oplus 1)S_\alpha^{-1}$
with $X_\alpha\prec P_\alpha$ and
$X_\alpha\prec N_\alpha$.

We claim that $d(X_\alpha,X_\beta)=4$. In fact, if a maximal matrix $M$ satisfies $M
\succ X_{\alpha}$, then,
 up to equivalence, either $M=N_\alpha$, or $M=P_\alpha$.
Hence, if $X_\alpha\path  Y_{\alpha,\beta}\path Z_{\alpha,\beta}\path X_{\beta}$ would be
a path of length three, connecting $X_\alpha$ and $X_\beta$ for $\alpha \ne \beta$, then
we may assume without loss of generality that $Y_{\alpha,\beta}$ and $Z_{\alpha,\beta}$
are maximal matrices (see the proof of Theorem~\ref{thm:d=4}). Hence $Y_{\alpha,\beta}$
is equivalent either to $N_{\alpha}$ or $P_{\alpha}$ and $Z_{\alpha,\beta}$ is equivalent
either to $N_{\beta}$ or $P_{\beta}$. This contradicts equation~\eqref{eq:n=4}, so
$d(X_\alpha,X_\beta)\ge 4$ for $\alpha\ne\beta$. Observe that one of the paths from
$X_\alpha$ to $X_\beta$ is $X_\alpha\path  N_\alpha \path  E_{11} \path N_\beta\path
X_\beta$.

{\bf Case \bm{$n\ge 5$}}. Let $A=\diag(\lambda_1,\dots,\lambda_n)$ where $
\lambda_i$ are
pairwise distinct. Consider an infinite family of rank one nilpotent
matrices $R_\alpha$ indexed by scalars  $\alpha\in\FF$:
$$R_\alpha=R+\alpha\tilde{R};\quad\; R=xf^{\tr},\tilde{R}= xg^{\tr},\; x=
\left[\begin{smallmatrix}
1 \\
\vdots\\[2mm]
1\\
\end{smallmatrix}\right],\;f=\left[\begin{smallmatrix}
2-n\\
1 \\
\vdots\\[2mm]
1\\
0
\end{smallmatrix}\right]
,\;g=\left[\begin{smallmatrix}
1-n\\
1 \\
\vdots\\[2mm]
1\\
1
\end{smallmatrix}\right],$$
 Note that $S_\alpha=\Id+R_\alpha
$ is invertible with
  $S_\alpha^{-1}=\Id-R_\alpha$, and that $S_\alpha^{-1}S_\beta
=(\Id-R_\alpha)(\Id+R_\beta) =\Id+(\beta-\alpha)\tilde{R}$. Let
us define for every $\alpha\in \FF$ the matrix
$X_\alpha=S_\alpha AS_\alpha^{-1}.$
We will prove that $d(X_\alpha,X_\beta)=4$ for $\alpha\ne\beta$.

Note first that the distance is invariant for simultaneous conjugation. So, we may
replace $(X_\alpha,X_\beta)$ with $(S_\alpha^{-1}X_\alpha S_\alpha,S_\alpha^{-1}X_\beta
S_\alpha)=(A,SAS^{-1})$, where $S=S_\alpha^{-1}S_\beta=\Id+(\beta-\alpha)\tilde{R}$.
Now, to prove $d(A,SAS^{-1})=4$ it suffices to show that, given any non-scalar matrices
$D_1\in \C(A)$ and $SD_2S^{-1}\in S\C(A)S^{-1}$, they do not commute.

By the choice of minimal $A$, $\C(A)$ consists of diagonal matrices
only, hence $D_1$ and
$D_2$ are diagonal. Assume erroneously that $D_1$ and $SD_2S^{-1}$ do
commute, i.e., that
$D_1(SD_2S^{-1})=(SD_2S^{-1})D_1$, or equivalently,
$$D_1(\Id+\tilde{x} g^{\tr})D_2(\Id-\tilde{x} g^{\tr})
=(\Id+\tilde{x} g^{\tr})D_2(\Id-\tilde{x} g^{\tr})D_1,$$
where  $\tilde{x}=(\beta-\alpha)x$.
Since diagonal matrices commute,   we get after expansion and
simplification
\begin{multline}\label{eq:star}
(D_1\tilde{x})(D_2g)^{\tr}-(D_1D_2\tilde{x})g^{\tr}-
(g^{\tr}D_2\tilde{x})\cdot(D_1\tilde{x}) g^{\tr}\\
=\tilde{x}(D_1D_2g-(g^{\tr}D_2\tilde{x})\cdot D_1g)^{\tr}-
(D_2\tilde{x})(D_1g)^{\tr}.
\end{multline}
  Notice  that an eigenvector  of a  non-scalar diagonal matrix has at least one nonzero entry.
Hence, $\tilde{x}=(\beta-\alpha)(1,\dots,1)^{\tr}$ and $g=(1-n,1,\dots,
1,1)^{\tr}$ can not be
eigenvectors of a non-scalar diagonal matrix. In particular, $g$ and
$D_2g$ are linearly
independent and so there exists a vector $y$ such that $g^{\tr}y=0$ and
$(D_2g)^{\tr}y=1$. Post-multiplying both sides  of equation~
\eqref{eq:star} with $y$, we
now have
$D_1\tilde{x}=\mu \tilde{x}+\nu D_2\tilde{x}$, $\mu=(D_1D_2g-
(g^{\tr}D_2\tilde{x}) D_1g)^{\tr}y$
and $\nu=-(D_1g)^{\tr}y$.
  We infer that $(D_1-\nu D_2)\tilde{x}=\mu \tilde{x}$, hence $(D_1-
\nu D_2)$ is a scalar matrix because
$\tilde{x}$ has all its entries nonzero.
Thus,
$D_1=\lambda\Id+\nu D_2$ for some $\lambda$. This
simplifies the starting
equation $D_1(SD_2S^{-1})=(SD_2S^{-1})D_1$ into
$$D(SDS^{-1})=(SDS^{-1})D;\qquad D=D_2,$$
wherefrom also the derived equation~(\ref{eq:star}) simplifies into
\begin{multline}\label{eq:multline2}
(D \tilde{x})(D g)^{\tr}-(D^2 \tilde{x})g^{\tr}-(g^{\tr}D \tilde{x})
\cdot(D \tilde{x})
g^{\tr}\\
=\tilde{x}(D^2 g-(g^{\tr}D \tilde{x})\cdot D g)^{\tr}-(D \tilde{x})(D
g)^{\tr}.
\end{multline}
  By the similar arguments as above we find a vector $z$ such that
$g^{\tr}z=0$ and
$(Dg)^{\tr}z=1$, and continuing along the lines we see that
$$D\tilde{x}=((D^2g)^{\tr}z-(g^{\tr}D \tilde{x}))\cdot \tilde{x}-D
\tilde{x}.$$
If $\mathrm{char}\,\FF\ne2$ then the above equation implies that  $
\tilde{x}$ is an
eigenvector of a diagonal matrix $D$ which is possible only when $D_2=D
$ is scalar, a
contradiction.

However, if $\mathrm{char}\,\FF=2$ then we choose a vector, still
named $z$, such that
$g^{\tr}z=1$ and $(Dg)^{\tr}z=0$. Similarly as above, this simplifies
equation~\eqref{eq:multline2} into
$$D^2\tilde{x}+(g^{\tr}D \tilde{x})\cdot
D\tilde{x}=\mu\tilde{x};\qquad\mu=(D^2{g})^{\tr}z.$$ Arguing as above, $D^2+(g^{\tr}D
\tilde{x})D-\mu\Id=0$. Thus, $D_2=D$, being non-scalar diagonal,  has exactly two
distinct eigenvalues: $d_1$ and $d_2$ (with multiplicities $k$ and $n-k $, respectively),
because it is annihilated by a quadratic polynomial $p(\lambda)= \lambda^2+(g^{\tr}D
\tilde{x})\lambda-\mu=(\lambda-d_1)(\lambda-d_2)$. With no loss of generality we assume
that $d_1=D_{1,1}$.  Then, comparing the coefficients in characteristics 2, gives
$d_1+d_2=(g^{\tr}D \tilde{x})=(\beta-\alpha) \bigl(
(1-n)d_1+(k-1)d_1+(n-k)d_2\bigr)=(\beta-\alpha)(n-k)(d_1+d_2)$. Since
$\mathrm{char}\,\FF=2$ and  $D$ is not a scalar matrix, we can divide by $d_1+d_2$ to
obtain $(\beta-\alpha)(n-k)=1$. Observe that in characteristic two, $(\beta-\alpha)(n-k)$
is either equal to 0 or $\beta-\alpha$ and since $(\beta-\alpha)(n-k)=1$, we have that
$\beta-\alpha=(\beta-\alpha)(n-k)=1$. and thus $\beta=\alpha+1$. Clearly, we can choose
an infinite subset of indices
$\mathfrak{A}=\{0,\alpha_1,\alpha_1+\alpha_2,\alpha_1+\alpha_2+
\alpha_3,\dots\}\subset\FF$ such that $\alpha-\beta\ne 1$ for $\alpha,\beta\in
\mathfrak{A}$. For this subset, $d(X_\alpha,X_\beta)=4$.\medskip

To prove the rest, observe that the rank-one matrix $$(\Id+ x (f+
\alpha
g)^{\tr})e_1e_1^{\tr}(\Id- x (f+\alpha g)^{\tr})=S_\alpha E_{11} S_
\alpha^{-1}$$
commutes with $X_\alpha=S_\alpha AS_\alpha^{-1}$. Now, since $n\ge 5$
there exists a
nonzero vector $w$ with $w^{\tr}e_1=0=w^{\tr} x =w^{\tr}f=w^{\tr}g$.
Then, a rank-one
matrix $Z=ww^{\tr}$ commutes with $S_\alpha E_{11} S_\alpha^{-1}$
which gives the path
$$X_\alpha\path  S_\alpha E_{11} S_\alpha^{-1}\path  Z.$$
Hence, $d(X_\alpha,Z)\le 2$ for every $\alpha$. Actually, no shorter
path exists, because
otherwise, we could join the shorter path for some $\alpha$ with the
above path for some
other index $\beta$ to obtain that $d(X_\alpha,X_\beta)\le 3$, a
contradiction.
\end{proof}

We next proceed with the classification of matrices which are equivalent to rank-one matrices.
In this classification we will need the following lemma.

\begin{lemma}\label{lem:nonminimal}
Let $n\ge 4$.
Suppose $A \in M_{n}(\FF)$ is
  \begin{enumerate}
   \item[(i)] either a maximal matrix with $2 \leq \rk A \leq n-2$, or
   \item[(ii)] a nilpotent matrix with $A^3=0$ and $\rk(A^2)=1$.
  \end{enumerate}
Then there exists a nonminimal  matrix  $X$, such that $d(A,X) \geq 3$.
\end{lemma}

\begin{proof}
(i)  As already observed in Preliminaries, a  maximal matrix $A$  is either a square-zero
matrix or an idempotent, up to equivalence. Let $k=\rk A$.

If $A$ is square-zero, then $2 \leq k \leq \frac{n}{2}$. We define
$s_{\ell}=(1,1,\dots,1)^{\tr}\in \FF^\ell$ and
 $z_{2\ell}= (0,1,0,1,\dots,0,1)^{\tr}\in \FF^{2\ell}$.
   Also, let $N_{2\ell}=\bigoplus_{i=1}^\ell J_2^{\tr} \in M_{2\ell}(\FF)$. Note that $N_{2\ell}^2=0$ and
  $\rk N_{2\ell} =\ell$.
It is easy to see that a matrix
\begin{equation} \label{eq:matrixA}
\begin{bmatrix}
       N_{2k-2} & 0 & z_{2k-2} \\
       0 & 0_{n-2k+1, n-2k+1} & s_{n-2k+1}\\
   0 &   0  & 0_{1, 1}
  \end{bmatrix}
\end{equation}
is a square-zero of  rank $k$, hence conjugate to $A$. So, we can assume without loss of generality
that $A$ is already in the form \eqref{eq:matrixA}.

  Next, let us define a matrix $X=J_2 \oplus 0_1 \oplus D$, where $D$ is a diagonal matrix
  with $n-3$ distinct nonzero diagonal entries. Clearly, $X$ is nonminimal.
   We will prove that $d(A,X)\geq 3$, i.e.,  any matrix that commutes with $A$ and $X$ is a scalar
   matrix.

  First, let us assume  $k=2$. It is easy to see that every matrix  $B\in\C(X)$ can be decomposed in the following way
  \begin{equation}\label{eq:commX}
   B= \left[ \begin{matrix}
    T & S_1 & 0_{2, 1} \\
    S_2 & D' & 0_{n-3, 1} \\
    0_{1, 2} & 0_{1, n-3} & \lambda
   \end{matrix}
   \right]\,
  \end{equation}
  where  $T= \left[ \begin{matrix}
   a & b  \\
   0 & a
  \end{matrix}
  \right] \in M_2(\FF)$, $D'=\mathrm{diag}(d_3,d_4,\ldots,d_{n-1})\in M_{n-3}(\FF)$, $S_1\in M_{2, n-3}(\FF)$ has the only
  nonzero entry in its upper left corner, $S_2 \in M_{n-3, 2}(\FF)$ has the only
  nonzero entry in its upper right corner, and $\lambda\in\FF$. Note that the blocks in the decomposition of $B$ correspond to the blocks in the decomposition of $A$.
  Suppose $B \in \C(X)$ also commutes with  $A= \left[ \begin{matrix}
       N_2 & 0 & z_2 \\
       0 & 0_{n-3, n-3} & s_{n-3}\\
       0 & 0& 0_{1, 1}
  \end{matrix}
  \right]$. Then, $N_2S_1=0$ and $S_2N_2=0$ imply that $S_1=0$ and $S_2=0$.
  Moreover, from $D's_{n-3}=\lambda s_{n-3}$ and $Tz_2=\lambda z_2$  we easily see that $D'=
  \lambda\Id_{n-3}$ and
  $T=\lambda\Id_2$. Thus, $B=\lambda\Id$, so $d(A,X)\geq 3$.

  Now, let us consider the case $k\geq 3$. Again we decompose every $B\in \C(X)$
  to the blocks that correspond to the block decomposition of $A$:
  \begin{equation*}
   B= \left[ \begin{matrix}
    B_1 & 0_{2k-2, n-2k+1} & 0_{2k-2, 1} \\
    0_{n-2k+1, 2k-2} & D' & 0_{n-2k+1, 1} \\
    0_{1, 2k-2} & 0_{1, n-2k+1} & \lambda
   \end{matrix}
   \right]\,
  \end{equation*}
  where  $B_1= \left[ \begin{matrix}
   a & b & c_1   \\
   0 & a & 0\\
   0 & c_2 & d_3
  \end{matrix}
  \right] \oplus \mathrm{diag}(d_4,d_5,\ldots,d_{2k-2}) \in M_{2k-2}(\FF)$ and
   $D' \in M_{n-2k+1}(\FF)$ is a diagonal
  matrix. Note that $B$ is the same as in~\eqref{eq:commX} but decomposed in a different way.
  Suppose $B$ also commutes with  $A$ as defined in \eqref{eq:matrixA}.
Similarly as before, we have $D's_{n-2k+1}=\lambda s_{n-2k+1}$ and thus
$D'=\lambda\Id_{n-2k+1}$.
  Moreover, it is straightforward that  from $B_1z_{2k-2}=\lambda z_{2k-2}$ and $N_{2k-2}B_1=B_1N_{2k-2}$
  we obtain $B_1=\lambda\Id_{2k-2}$. This completes the proof that  $d(A,X)\geq 3$.

 If $A$ is an idempotent, then  $\rk(\Id-A)=n-\rk A$. Since $A$ and
    $\Id-A$ are equivalent, we can thus assume
 without loss of generality that $A$ is of rank $k$ with  $ \frac{n}{2} \leq k \leq n-2$. Let
 $W$ be a $k \times (n-k)$ matrix with the only nonzero elements being
  $$W_{1,n-k}=W_{2,n-k-1}=W_{3,n-k-2}=\ldots=W_{n-k,1}=W_{k,1}=W_{k,n-k}=1\, .$$
  Note that, if $k=\frac{n}{2}$,  the rows $k$ and $n-k$ coincide. Using an appropriate conjugation we can additionally assume that
$$  A= \begin{bmatrix}
   \Id_k & W\\
   0_{n-k, k} & 0_{n-k, n-k}
  \end{bmatrix}.$$
Let us define a nonminimal matrix $X=J_k \oplus 0_1 \oplus \Id_{n-k-1}$. We
 will prove that $d(A,X)\geq 3$, i.e.,   any matrix $B \in \C(A) \cap \C(X)$ is a
 scalar matrix. It is a straightforward calculation that
 $$\C(A)=\left\{
  \left[ \begin{matrix}
   M & MW-WN\\
   0_{n-k,k} & N
  \end{matrix}
  \right]; \;  M \in M_{k}(\FF), N \in M_{n-k}(\FF)
 \right\}$$
 and that $\C(X)$ consists of all matrices of the form
  $$ \left[ \begin{matrix}
   U & s & 0_{k, n-k-1}\\
   v^{\tr} & \lambda & 0_{1, n-k-1}\\
   0_{n-k-1, k}&   0_{n-k-1, 1}  & Y
  \end{matrix}
  \right],$$
  where  $U$ is an upper triangular Toeplitz $k \times k$ matrix,
  $s=  (s_1 , 0 , \ldots ,0  )^{\tr} \in\FF^k$,
  $v=  ( 0 , \ldots ,0 , v_k )^{\tr}\in\FF^k$, $Y = \big[y_{ij}\big]_{2\le i,j\le n-k}\in M_{n-k-1}(\FF)$, and $\lambda\in\FF$.

Suppose $B=\left[ \begin{matrix}
   M & MW-WN\\
   0 & N
  \end{matrix}
  \right] \in \C(A) \cap \C(X)$. It follows that $M= \sum_{i=1}^k m_i J_k^{i-1}$,
  $N=\lambda \oplus Y$ and $(MW-WN)_{ij}=0$ except possibly for $i=j=1$.

 By equations
 \begin{align*}
 0&=(MW-WN)_{k,1}=m_1-\lambda, \\
 0&=(MW-WN)_{k,n-k}=m_1-y_{n-k,n-k},\\
 0&=(MW-WN)_{i,n-k}=m_{k-i+1}\qquad \hbox{ for all $i$ with }(k-1)\ge i\ge (n-k)
 \end{align*}
 it follows that
 $m_1=y_{n-k,n-k}=\lambda $ and $m_2=m_3=\ldots=m_{2k-n+1}=0$. Moreover, by
 $0=(MW-WN)_{i,1}=m_{k-i+1}$ for $i=2,3,\ldots,n-k-1$, it follows that
 $m_{2k-n+2}=\ldots=m_{k-1}=0$. Now, equation $0=(MW-WN)_{1,n-k}=m_{k}$ completes the
 proof that $M=\lambda  \Id_k$.

 We proceed by
  $0=(MW-WN)_{i,n-k-i+1}=m_1-y_{n-k-i+1,n-k-i+1}$ for  $i=2,\ldots,n-k-1$ and
  $0=(MW-WN)_{i,j}=-y_{n-k-i+1,j}$ for  $i=1,2,\ldots,n-k-1$ and $j=2,3,\ldots,n-k$,
 such that $i+j \ne n-k+1$. It follows that $N=\lambda\Id_{n-k}$ and $(MW-WN)=0$.
 Thus, $B=\lambda \Id$ and $d(A,X) \geq 3$.

(ii) Let $A$ be a nilpotent matrix such that $A^3=0$ and $\rk(A^2)=1$. We may assume $A$
is already in its Jordan canonical form, i.e.,
$$A=J_3\oplus \bigoplus_{i=1}^kJ_2\oplus 0_{n-3-2k}.$$
The centralizer of $A$ is contained in the set of matrices of the form
 $B=\left[ \begin{matrix}
   T & S_1 \\
   S_2 & V
  \end{matrix}
  \right] $, where $T=t_0 \Id_3+t_1J_3+t_2J_3^2 \in M_3(\FF)$,
  $V  \in M_{n-3}(\FF)$, and where the first column of
  $S_2  \in M_{n-3, 3}(\FF)$  as well as the last row of
  $S_1  \in M_{3, n-3}(\FF)$ contain only zero entries.

 Now, let us define the nonminimal matrix $X=1 \oplus 0 \oplus J_{n-2}$ and take any
 $B \in \C(A) \cap \C(X)$. Since $B\in \C(X)$, its  off-diagonal entries on the first row
 and the first column
 are all zero. Comparing with the above form for $B$ we deduce that $T=t_0 \Id$. Moreover,
 $B\in \C(X)$ also implies that the bottom-right $(n-2) \times (n-2)$ block of $B$ is
 upper triangular Toeplitz matrix, which is moreover equal to $t_0' \Id_{n-2}$ for some $t_0'\in\FF$ by the fact that the third row of
 $S_1$ vanishes. Actually, $t_0=t_0'$ because a $3\times 3$ block $T$ overlaps with
 $(n-2)\times (n-2)$ bottom right block. Further, $B\in\C(X)$ implies that the only
 possible off-diagonal nonzero entries in the second row and column lie at positions $(2,n)$, and
 $(3,2)$. Actually, $B_{32}=T_{32}=0$, while from $B\in \C(A)$ we deduce that if
 $B_{2n}\ne0$ then also $B_{1(n-1)}\ne0$, which would contradict the fact that the first row
 of $B$ has zero off-diagonal entries. Hence, $B_{2n}=B_{32}=0$ and so $B= t_0\Id$ is a
 scalar and therefore $d(A,X) \geq 3$.
\end{proof}

\begin{theorem}\label{thm:classification-of-rk1}
The following statements are equivalent for a non-scalar matrix $R$.
\begin{itemize}
\item [(i)] $R$ is equivalent to a matrix of rank one.
\item[(ii)] $d(R,X)\le 2$ for every nonminimal matrix $X$.
\end{itemize}
\end{theorem}

\begin{proof}
 If $n=3$, then every nonminimal matrix is equivalent to a rank-one matrix,
 so we may assume that $n\ge4$.

 To prove that $(i)\Longrightarrow (ii)$, we can assume without loss of generality that $\rk R=1$.
 Let  $X$ be an arbitrary nonminimal matrix. Then $d(R,X)\le 2$ by
 Lemma~\ref{lem:A'capR'}.

 $\neg(i)\Longrightarrow \neg(ii)$. Suppose $R$ is not equivalent to a rank-one matrix.
 Note that there exists at least one maximal matrix $M\succ R$. In fact,  $M=p(R)$
 for some polynomial $p$. Moreover, we can assume that every maximal $M\succ R$ is either a nonzero square-zero matrix
 or a non-scalar idempotent. Hence $1 \le \rk M \le n-1$. Note that $\rk M = n-1$ implies $M$ is an idempotent and therefore it
 is equivalent to a maximal matrix of rank one. So we can assume that $1 \le \rk M \le n-2$.

 If for a maximal $M\succ R$ we have $2\le \rk M \le n-2$, then by Lemma
 \ref{lem:nonminimal}   there exists a nonminimal matrix $X$ with $d(M,X)\ge 3$. Hence also
  $d(R,X)\ge 3$ because $\C(R)\subseteq\C(M)$.

 Otherwise, every maximal matrix $M\succ R$ is equivalent to a rank-one matrix. This
 implies that (i) $R$ is  either equivalent to a nilpotent matrix with exactly one Jordan block of
 dimension $3$ and all other
 cells of dimension at most $2$, or (ii) $R$ is equivalent to a matrix whose Jordan structure is equal to
  $1 \oplus J_2 \oplus 0_{n-3}$, or (iii) $R$ is equivalent to a matrix whose Jordan structure is equal to $1 \oplus J_3\oplus\bigoplus_{i=1}^k J_2 \oplus 0_{n-3-2k}$.
 In the first case, Lemma~\ref{lem:nonminimal} assures that there exists
  a nonminimal $X$ with $d(R,X)\ge 3$. In the case (ii) we have, modulo conjugation, $R=1 \oplus J_2 \oplus 0_{n-3}$. It is easy to see that
  $X=J_2\oplus J_{n-2}$ is nonminimal and $d(R,X)\ge 3$. In case (iii) we have, modulo conjugation,
  $R\prec R'=0 \oplus J_3\oplus\bigoplus_{i=1}^k J_2 \oplus 0_{n-3-2k}$. Again, Lemma~\ref{lem:nonminimal}
  gives a nonminimal matrix~$X$ with $d(R',X) \geq 3$, so also $d(R,X)\ge 3$.
\end{proof}

In the previous theorem rank-one matrices are classified with the help of matrices which
are not minimal. We next classify minimal matrices as the ones which maximize the
distance in a commuting graph.
\begin{theorem}
The following are equivalent for a matrix $A\in M_n(\FF)$.
\begin{itemize}
\item [(i)] $A$ is minimal.
\item[(ii)] There exists a matrix $X$ such that $d(A,X)=4$.
\end{itemize}
\end{theorem}
\begin{proof}
$(i)\Longrightarrow(ii)$. This follows from Theorem~\ref{thm:d=4}.

$\neg(i)\Longrightarrow\neg(ii)$.
Let~$A$ be a non-minimal matrix, and let $X$ be any matrix. By Lemma~\ref{lem:rk-1'=all},
there exists a rank-one matrix
$R$ with $d(X,R)\le 1$.
By Theorem~\ref{thm:classification-of-rk1} we have $d(A,R)\le2$, so triangle inequality gives \mbox{$d(A,X)\le 3$}.
\end{proof}

\begin{remark}
Combining the previous two theorems yields that $R$ is equivalent to rank-one matrix
if and only if $d(R,X)\le 2$ for every  matrix $X$ such that $d(X,Z)\le 3$,
for all $Z\in M_n(\FF)$.
\end{remark}

Semisimple matrices can also be classified using the distance in the commuting graph. Before doing that we need two lemmas.
\begin{lemma}\label{lem:diagonalizable1}
Suppose a minimal matrix $B\in M_n(\FF)$ is semisimple. Then
for
any $Y\path X \path B$ there exists a minimal matrix $M$ with $Y\path  M\path  X$.
\end{lemma}
\begin{proof}
Assume with no loss of generality that $B$ is diagonal. Then,
every $X\in \C(B)$ is also diagonal. Using simultaneous
conjugation on $(B,X)$ we may further assume that
$X=\lambda_1\Id_{n_1}\oplus\dots\oplus\lambda_k \Id_{n_k}$, with
$\lambda_1,\dots,\lambda_k$ pairwise distinct and
$n_1,\dots,n_k\ge1$. Now, since $Y$ commutes with $X$ we have that
$Y\in \C(X)=M_{n_1}(\FF)\oplus\dots\oplus M_{n_k}(\FF)$.
Consequently, $Y=Y_1\oplus \dots\oplus Y_k$ is block-diagonal and
we may find an invertible block-diagonal matrix
$S=S_1\oplus\dots\oplus S_k$ such that $S^{-1}XS=X$ and
$S^{-1}YS=\bigoplus_{i=1}^k S_i^{-1}Y_iS_i$ is in Jordan
upper-triangular form; say $S^{-1}YS=\bigoplus_{i=1}^s
J_{m_i}(\mu_i)$, with $m_i\ge 1$, $s \ge k$. Then we can choose
distinct $\nu_1,\dots,\nu_s\in\FF$, such that the matrix
$M=S\bigoplus_{i=1}^s J_{m_i}(\nu_i)S^{-1}$ is neither equal to $X$ nor $Y$.
Also, since $\nu_1,\dots,\nu_s$ are distinct, $M$ is nonderogatory,
hence minimal, and it commutes with $X$ and with $Y$.
\end{proof}

\begin{lemma}\label{lem:diagonalizable2}
 Suppose a minimal $B\in M_n(\FF)$ is not semisimple. Then there exist
matrices $X,Y$ with $Y\path X\path  B$, but such that no minimal matrix commutes with
both $X$ and $Y$.
\end{lemma}
\begin{proof}
With no loss of generality assume $B$ is already in its upper-triangular Jordan form,
$B=J_{n_1}(\lambda_1)\oplus\dots\oplus J_{n_k}(\lambda_k)$ with $\lambda_1, \ldots, \lambda_{n_k}$ distinct
and $n_1\ge2$. Define
$X=E_{1n_1}=J_{n_1}^{n_1-1}\oplus 0_{n-n_1}\in \C(B)$ and for an arbitrary
$k\in\{1,\dots,n\}\setminus\{1,n_1\}$ define $Y=E_{1k}$. Clearly, $X$ commutes with $Y$.
Let us show that no minimal $A$ commutes with both $X$ and $Y$. Assume $A=\bigoplus_{j=1}^s J_{n_j}(\mu_j)$ is already in its Jordan canonical form.
Since $X\in \C(A)$ is of rank one, it follows that $X\in\FF J_{n_{j_1}}^{n_{j_1}-1}$
for some $j_1$, and likewise  $Y\in\FF J_{n_{j_2}}^{n_{j_2}-1}$ for some $j_2$. However, $\rk(X+Y)=1$ and so $j_1=j_2$,
which gives $X$ and $Y$ must be linearly dependent, a contradiction.
\end{proof}

\begin{theorem}
Let $A\in M_n(\FF)$ be a non-scalar matrix. Then the following are
equivalent.
\begin{itemize}
\item [(i)] $A$ is semisimple.
\item[(ii)]  There exists a minimal $B\in\C(A)$ such that for any $Y\path     X\path
    B$ there exists a minimal matrix $M$ with $Y\path M\path X$.
\end{itemize}
\end{theorem}
\begin{proof}
$(i) \Longrightarrow (ii)$. Assume without loss of generality that $A$ is already
diagonal. Choose distinct scalars $\mu_1,\dots,\mu_n$ to form a minimal matrix
$B=\diag(\mu_1,\dots,\mu_n)$ which clearly commutes with $A$.  Then, (ii) follows from
Lemma~\ref{lem:diagonalizable1}.

 $\neg(i) \Longrightarrow \neg(ii)$. Choose any minimal $B$
which commutes with non-semisimple $A$ (at least one does exist, for example, if
$A=S\bigoplus_{i=1}^k J_{n_i}(\lambda_i)S^{-1}$, then for distinct scalars
$\lambda_1,\dots,\lambda_k$ matrices $A$ and $B=S\bigoplus_{i=1}^k J_{n_i}(\lambda_i)S^{-1}$ commute).
Since $\C(B)=\FF[B]$ it follows that $A\in \FF[B]$ which implies that $B$ itself is not semisimple. It now follows
from Lemma~\ref{lem:diagonalizable2} that there exist $X,Y$ with $Y\path X\path B$, but no
minimal matrix commutes with both of them. So (ii) does not hold.
\end{proof}

Let us conclude  with an example of a connected commuting graph  over  algebraically non-closed field
with the diameter strictly larger than~$4$.
\begin{example} The commuting graph for $M_9(\ZZ_2)$ is connected with
diameter at least $5$.

Note that $\ZZ_2$ permits only one field extension  of degree $n=9$, and this is the
Galois field $GF(2^9)$ which  contains $GF(2^3)$ as the only proper intermediate
subfield. So, by~\cite[Theorem 6]{Akbari_Bidkhori_Mohammadian} the commuting graph of
$M_9(\ZZ_2)$ is connected. To see that its diameter is at least $5$, consider an
irreducible polynomial
$m(\lambda)=\lambda^9+\lambda^8+\lambda^4+\lambda^2+1\in\ZZ_2[\lambda]$ and let
$\widehat{A}=C(m)\in M_9(\ZZ_2)$ be its companion matrix. Since $\widehat{A}$ has a
cyclic vector, $\C(\widehat{A})=\ZZ_2[\widehat{A}]$ by a well known Frobenius result on
dimension of centralizer~(see for example
\cite[Corollary~1]{Akbari_Bidkhori_Mohammadian}), and this is a field extension of
$\ZZ_2$~\cite[Theorem 4.14, pp. 472]{Joshi}  of index $n=9$. Actually, $\C(\widehat{A})$
is isomorphic to $GF(2^9)$ by the uniqueness of field extensions for finite fields. In
the sequel we will identify the two.

Since the field extension $\ZZ_2\subset GF(2^9)$ contains only $GF(2^3)$ as a proper
intermediate subfield, we see that each  $X\in\C(\widehat{A})\setminus GF(2^3)$ satisfies
$\ZZ_2[X]=\ZZ_2[\widehat{A}]=\C(\widehat{A})$ and in particular $X$ and $\widehat{A}$ are
polynomials in each other so they are equivalent. Moreover, each non-scalar
$\widehat{Y}\in GF(2^3)$ satisfies $\ZZ_2[\widehat{Y}]=GF(2^3)$, because no proper
intermediate subfields exist between $\ZZ_2\subset GF(2^3)$, and in particular,
$\C(\widehat{Y}_1)=\C(\widehat{Y}_2)$ for any two non-scalar
$\widehat{Y}_1,\widehat{Y}_2\in GF(2^3)\subset GF(2^9)=\C(\widehat{A})$.

There exists a polynomial $p$ so that $\widehat{Y}=p(\widehat{A})\in GF(2^3)\setminus\{0,1\}$. As
the field $GF(2^3)$ contains no idempotents other than $0$ and $1$ we see that the
rational canonical form of $\widehat{Y}$ consists only of cells which correspond to powers of the same
irreducible polynomials. Likewise, the field contains no nonzero nilpotents, so each cell
of $\widehat{Y}$ corresponds to the same irreducible polynomial, raised to power $1$.
Moreover, $GF(2^3)$ has no subfields other that $\ZZ_2$, so $\ZZ_2[\widehat{Y}]=GF(2^3)$ and
hence the minimal polynomial of $\widehat{Y}\in GF(2^3)$ has degree $[GF(2^3):\ZZ_2]=3$. This
polynomial is relatively prime to its derivative, so in a splitting field, $\widehat{Y}$ has
three distinct eigenvalues. It easily follows that $\widehat{Y}$ is conjugate to a matrix
$C\oplus C\oplus C$, with $C$ being a $3\times 3$ companion matrix of some irreducible
polynomial of degree $3$. Let $S_1$ be an invertible matrix such that
$\widehat{Y}=S_1^{-1}(C \oplus C \oplus C)S_1$ and define
$$A=S_1\widehat{A}S_1^{-1}.$$ Clearly then $p(A)=S_1\widehat{Y}S_1^{-1}=C\oplus C\oplus C$ and
it  follows that
\begin{equation}\label{eq:exa}
\C(p(A))=\begin{bmatrix}
\ZZ_2[C] &\ZZ_2[C] &\ZZ_2[C]\\
\ZZ_2[C] &\ZZ_2[C] &\ZZ_2[C]\\
\ZZ_2[C] &\ZZ_2[C] &\ZZ_2[C]
\end{bmatrix}.
\end{equation} Since $\ZZ_2[\widehat{Y}]=GF(2^3)$ we obtain $\ZZ_2[C]=GF(2^3)$.

 Consider a $3\times 3$ block
matrix
\begin{equation*}\label{eq:N}
N=\begin{bmatrix}
E_{13} &0 &0\\
0 &0 &E_{13}\\
E_{32} & 0 &0
\end{bmatrix}, \quad E_{13},E_{13},E_{32}\in M_3(\ZZ_2).
\end{equation*}
It is immediate that
$N^3=0$, so $\Id+N$ is invertible. Define
$$B=(\Id+N)A(\Id+N)^{-1}.$$
We will show that $d(A,B) \ge 5$.

 Suppose there exists a path $A \path V \path Z \path W \path B$ of length 4.
 Note that $V\in
GF(2^3)\subset\C(A)$. Otherwise, if  $V\in \C(A)\setminus GF(2^3)$ then $\C(V)=\C(A)$ and
such $V$ has exactly the same neighbours as $A$. Since $B=(\Id+N)A(\Id+N)^{-1}$, it follows
$W=(\Id+N)U(\Id+N)^{-1}$ for some $U\in GF(2^3)\subset\C(A)=(\Id+N)^{-1}\C(B)(\Id+N)$.
 Recall that any two non-scalar elements in $GF(2^3)$ have
the same centralizer. So in particular we might take $U=V=p(A)=C\oplus C\oplus C$ where
polynomial $p$ was defined before. For any $Z\in \C(V)\cap\C((\Id+N)V(\Id+N)^{-1})$ we
have
$$Z=(\Id+N)\widehat{Z}(\Id+N)^{-1},\qquad Z, \widehat{Z}\in \C(V)$$
and hence, by postmultiplying with $(\Id+N)$ and rearranging,
\begin{equation}\label{eq:exa-2ndeq}
Z-\widehat{Z}=N\widehat{Z}-ZN.
\end{equation}
Let us write $Z=\bigl[Z_{ij}  \bigr]_{1\le i,j\le3}$ and $\widehat{Z}=\bigl[\widehat{Z}_{ij}
\bigr]_{1\le i,j\le3}$ as $3\times 3$ block matrices and by \eqref{eq:exa} we have that
 $Z_{ij},\widehat{Z}_{ij}\in\ZZ_2[C]=GF(2^3)\subseteq
M_3(\ZZ_2)$, hence each of them is either zero or invertible. Then
(\ref{eq:exa-2ndeq}) implies
$$\bigl[Z_{ij}-\widehat{Z}_{ij}   \bigr]_{ij}=
\begin{bmatrix}
-Z_{11}E_{13}-Z_{13}E_{32}+E_{13}\widehat{Z}_{11}&E_{13}\widehat{Z}_{12}&E_{13}\widehat{Z}_{13}-Z_{12}E_{13}\\
-Z_{21}E_{13}-Z_{23}E_{32}+E_{13}\widehat{Z}_{31}&E_{13}\widehat{Z}_{32}&E_{13}\widehat{Z}_{33}-Z_{22}E_{13}\\
-Z_{31}E_{13}-Z_{33}E_{32}+E_{32}\widehat{Z}_{11}&E_{32}\widehat{Z}_{12}&E_{32}\widehat{Z}_{13}-Z_{32}E_{13}
\end{bmatrix}.
$$
Observe that each block on the left side belongs to $\ZZ_2[C]=GF(2^3)\subseteq
M_3(\ZZ_2)$, and so is either zero or invertible. On the other
hand, on the right side, each block in the last two columns  has rank at most two. We
deduce that  the last two columns on both sides are zero. In particular, comparing  the
second columns we see that $\widehat{Z}_{12}=Z_{12}=0$ and $\widehat{Z}_{32}=0$, so
$Z_{22}=\widehat{Z}_{22}$, and $Z_{32}=\widehat{Z}_{32}=0$. Putting this in the
above equation and simplifying,  the last column then gives $\widehat{Z}_{13}=0$, so
$Z_{13}=\widehat{Z}_{13}=0$, $Z_{23}=\widehat{Z}_{23}$, and
$Z_{33}=\widehat{Z}_{33}$. Also, comparing the   $(2,3)$ positions gives
$$0=Z_{23}-\widehat{Z}_{23}=E_{13}\widehat{Z}_{33}-\widehat{Z}_{22}E_{13}=e_1({\widehat{Z}_{33}}^{\tr}e_3)^{\tr}-
\widehat{Z}_{22}e_1e_3^{\tr}.$$  Moreover,
$\widehat{Z}_{33}^{\tr}e_3=\lambda e_3$ and $\widehat{Z}_{22}e_1=\lambda e_1$, $\lambda \in \ZZ_2$. Since
$\widehat{Z}_{33},\widehat{Z}_{22}\in\ZZ_2[C]$ and every vector is cyclic for $C$ we see
that $\widehat{Z}_{33}=\widehat{Z}_{22}=\lambda\Id_3$.
The matrix equation therefore simplifies to
$$\begin{bmatrix}
 Z_{11}-\widehat{Z}_{11} & 0 & 0 \\
 Z_{21}-\widehat{Z}_{21} & 0 & 0 \\
 Z_{31}-\widehat{Z}_{31} & 0 & 0
\end{bmatrix}=\begin{bmatrix}
   -Z_{11} E_{13} + E_{13}\widehat{Z}_{11} & 0 & 0 \\
    -Z_{21} E_{13}-\widehat{Z}_{23} E_{32} + E_{13} \widehat{Z}_{31}& 0 & 0   \\
  -Z_{31} E_{13}-\lambda  E_{32} +   E_{32}\widehat{Z}_{11} & 0 & 0
\end{bmatrix}.$$
Comparing the position $(1,1)$ gives by similar arguments as above that
$\widehat{Z}_{11}=Z_{11}=\mu\Id_3$. Inserting this into the equation we see after
rearrangement that the rank of the block at position $(3,1)$ is equal to $\rk((\mu -\lambda )
E_{32}-Z_{31}E_{13})\le 2$, which forces the two blocks at position $(3,1)$ to be
zero, i.e., $Z_{31}-\widehat{Z}_{31}=0=(\mu -\lambda ) E_{32}-Z_{31}E_{13}= (\mu
-\lambda ) e_3e_2^{\tr}-Z_{31}e_1e_3^{\tr}$. We immediately get
$Z_{31}=\widehat{Z}_{31}=0=(\mu -\lambda )$. Therefore,
$Z_{11}=Z_{22}=Z_{33}=\lambda\Id_3$. Finally, comparing the $(2,1)$
positions gives
$$Z_{21}-\widehat{Z}_{21}=-Z_{21}E_{13}-\widehat{Z}_{23}E_{32},$$
and arguing as above, $Z_{21}=\widehat{Z}_{21}=0$. Hence, $Z$ is scalar.
So, $\C(V)\cap \C(W)$ contains only scalars,
which gives that $d(A,B)\ge 5$.

\end{example}

\end{document}